\documentclass[12pt]{amsart}
\usepackage{amscd}
\usepackage[all]{xy}
\usepackage[notcite,notref]{showkeys}

\newcommand{\Mot}{{\operatorname{Mot}}}

\newcommand{\Hom}{\operatorname{Hom}}

\newcommand{\CH}{\operatorname{CH}}
\newcommand{\Br}{\operatorname{Br}}

\newcommand{\Corr}{\operatorname{Corr}}

\newcommand{\rank}{\operatorname{rank}}
 
\newcommand{\Pic}{\operatorname{Pic}}
\newcommand{\Spec}{\operatorname{Spec}}

\newcommand{\Bir}{\operatorname{Bir}}
\renewcommand{\Mot}{\operatorname{Mot}}

\newcommand{\C}{\mathbf{C}}
\renewcommand{\L}{\mathbf{L}}
\renewcommand{\P}{\mathbf{P}}

\newcommand{\N}{\mathbf{N}}
\newcommand{\Q}{\mathbf{Q}}
\newcommand{\Z}{\mathbf{Z}}
\newcommand{\un}{\mathbf{1}}

\newcommand{\sA}{\mathcal{A}}
\renewcommand{\v}{\mathbf{v}}
 \newcommand{\sC}{\mathcal{C}}
\newcommand{\sD}{\mathcal{D}}
\newcommand{\sY}{\mathcal{Y}}

\newcommand{\sM}{\mathcal{M}}

\newcommand{\sO}{\mathcal{O}}

\newcommand{\sJ}{\mathcal{J}}

 \newcommand{\sF}{\mathcal{F}}

 \numberwithin{equation}{section}

\theoremstyle{plain}
\newtheorem{thm}[equation]{Theorem}
\newtheorem{prop}[equation]{Proposition}
\newtheorem{lm}[equation]{Lemma}
\newtheorem{cor}[equation]{Corollary}
\newtheorem{conj}[equation]{Conjecture}

\theoremstyle{definition}

\newtheorem{ex}[equation]{Example}
\newtheorem{rk}[equation]{Remark}

\begin{document}

 \title{The  Chow motive  of LSV hyper-K\"alher tenfolds}
 \author[C. Pedrini]{Claudio Pedrini}
\address{Dipartimento di Matematica \\ %
Universit\'a degli Studi di Genova \\ %
Via Dodecaneso 35 \\ %
16146 Genova \\ %
Italy}
\email{claudiopedrini4@gmail.com}

\begin{abstract}  Let $X$ be a smooth cubic fourfold  over $\C$ and let $\pi : \sJ_U \to U$, with $ U \subset (\P^5)^*$, be the Lagrangian fibration whose   fibres   are the smooth hyperplane sections $Y_ H = X \cap H$, with $H \in U$. 
There always exists a (not unique) smooth compactification $\bar \sJ \to  (\P^5)^*$ which is a hyper-K\"alher manifold of OG10 type. Since two different compactifications  are birationally equivalent their Chow motives are isomorphic. 
For  a general $X$ a geometrical construction of a  smooth compactification $\sJ(X)$ with irreducible   fibres  has been described in [LSV]. In this note we prove that the Chow motive $h(\sJ(X)) $ is a direct summand of the (twisted) motive of $X^5$  and  therefore is  is of abelian type if $h(X)$ is of abelian type.We describe a 10 -dimensional family $\sF$ of cubics  $X$ such that the compactification  $\sJ(X)$ is unique, smooth, with irreducible fibres, and  the  Chow motive $h( \sJ(X) )$ is  of abelian type.\end{abstract}

\maketitle 

\section{Introduction}

A hyper-K\"ahler manifold (HK for short) is a simply connected  compact K\"ahler manifold $X$ such that $H^0(X, \Omega^2_X) =\C \omega$, where $\omega$ is a holomorphic 2-form on $X$ which is nowhere degenerate 
(as a skew symmetric form on the tangent space).\par
 In dimension 2 a HK manifold is a K3 surface  For every even complex dimension, there are two known deformations types of irreducible holomorphic symplectic
varieties: the Hilbert scheme $S^{[n]}$ of n-points on a K3 surface $S$ and the generalized Kummer
varieties. A generalized Kummer variety $X$ is of the form $X=K^n(A)=a^{-1}(0)$, where $A$ is an abelian surface and $a:  A^{[n+1]} \to A$ is the Albanese map.  In addition to these two series of examples, there are two exceptional examples. 
They occur in dimension 6 and in dimension 10 and were
discovered by O' Grady by resolving symplectically two singular
moduli space of sheaves on symplectic surfaces. These two examples are usually referred
to as OG6 and OG10. The hyper-K\"alher varieties of OG10 type are not deformation equivalent to $S^{[5]}$.\par
We will denote by $\sM_{rat}(\C)$ the (covariant) category of Chow motives over $\C$ (with $\Q$-coefficients) , by $\sM_{hom}(\C)$ the category of homological motives and by $\sM_A(\C)$ the category of Andre'  motives,.
The category $\sM_A(\C)$  is obtained from $\sM_{hom}(\C)$ by formally adjoining the Lefschetz involutions  $*_L$ associated to the Lefschetz isomorphisms $H^i(X) \simeq H^{2d-i}(X)$, where
$X$ is a smooth projective variety of dimension $d$, see [An]. The involutions $*_L$ are given by an algebraic correspondence if and only if the   Lefschetz standard conjecture $B(X)$ holds true.Therefore under $B(X)$ the category of Andre' motives coincides with $\sM_{hom}(\C)$. There  are functors

$$\CD \sM_{rat}(\C)@>{F}>>\sM_{hom}(\C)@>{G}>>\sM_A(\C)\endCD$$

Assuming Kimura' s Conjecture on the finite dimensionality of motives the functor  $F$ is conservative, i.e. it preserves isomorphism, while, under the standard Conjecture $B(X)$,  the functor $G$ is an equivalence of categories. Since the motives of K3 surfaces and of cubic 4-folds are of abelian type in the Andre' category  $\sM_A(\C)$ (see [An] ) they are conjecturally of abelian type in $\sM_{rat}(\C)$, i.e. they  lie  in the subcategory  $\sM^{Ab}_{rat}(\C)$ generated by abelian varieties.  For a K3 surface  $S$ this result is known only when $\Pic S$ has rank $\ge 19$, for Kummer surfaces and in some other scattered cases. For a cubic fourfold $X$ it is known  for a family of special cubic fourfolds in $\sC_d$, see [ABP]. A family of non-special cubic fourfolds with a motive of abelian type  has been described in in [Lat ]. By the results in  [deC-M] the same conjecture holds for  $h(S^{[n]})$ and hence  for a HK variety that is birational to $S^{[n]}$, because HK varieties that are birationally equivalent have isomorphic motives in $\sM_{rat}(\C)$.\par
 In [FFZ, Cor.1.15] it is proved that the Andre' motive of all the known projective hyper-K\"ahler varieties is of abelian type. All these results suggest the following conjecture:
\begin{conj}  The motive of a HK manifold  is of Abelian type in $\sM_{rat}(\C)$.\end{conj}
Let us recall what we mean by {\it LSV tenfolds}. Let $X \subset \P^5_{\C}$ be a smooth cubic  fourfold and let $U \subset B$,  with $B = (\P^5)^*$, the open subset parametrizing hyperplanes
$H \subset \P^5$  such that the cubic 3-fold  $Y_H =X \cap H$ is smooth. Let
$\pi: \sJ_U \to U$ be the Lagrangian fibration whose fiber over the point  $H \in U$ is the intermediate Jacobian $J(Y_H)$. By [LSV] and [Sa] there exists a projective HK manifold of OG10 type $\bar \sJ$ which compactifies $\sJ_U$
and which is equipped with a Lagrangian  fibration $\bar \sJ \to B$ extending the intermediate Jacobian fibration. A geometrical construction of a  smooth compacticficattion $\sJ(X)$ has been described in [LSV] in the case of a general $X$.   Here  "general" means  outside a countable union of divisors of the projective space of  cubic fourfolds , see [Bro, Cor 4.3] . In [Sa] this is extended to any cubic fourfold $X$, smooth or mildly singular. This construction loses the explicit geometry of special fibers of $\pi: \bar \sJ \to B$,  as described in [LSV]. In particular these fibers may be not irreducible.\par
 The Lagrangian fibered HK manifolds compactifying $\sJ_U \to U$ will be referred as {\it LSV tenfolds.} If $X$ is general  the compactification $\sJ(X)$ constructed in [LSV] will be referred to as a {\it general LSV tenfold.} The HK manifolds constructed in [Sa] are specialization of general LSV tenfolds.\par
In [Bro, Cor.4.9] an example is given of a smooth cubic fourfold $X$ such that there is no flat regular compactification of $\sJ_U \to U$ with irreducible fibers \par
In [DM,Thm.1 ]it is proved  that for  a  very good cubic fourfold $X$  there exists  a  smooth compactification of  $J_{U} \to U $ with  irreducible fibers.  Assuming $X$ is a very good cubic fourfold is equivalent to assuming that $X$ does not contain a plane, a rational normal cubic scroll or have a hyperplane section with a corank 3 singularity.\par
If $X$ is a smooth cubic fourfold two  smooth compactification  $\bar \sJ$ and $\bar \sJ'$ of  of $\sJ_U$  are birational equivalent and hence have isomorphic motives. In particular, for a general $X$ the motive of the  compactification $\sJ(X)$ in   [LSV] is isomorphic to the motive of any  other smooth compactification.\par
In [ACLS,Thm.1.2] it is proved that the LSV manifolds associated to a smooth cubic fourfold $X$ satisfy the  Lefschetz standard conjecture $B(X)$. Therefore their motives are of abelian type in the category $\sM_{hom}(\C)$ of homolological motives.\par
 In this note we show that  the motive of  a general LSV  tenfold $\sJ(X)$ belongs to the pseudo-abelian tensor 
subcategory generated by the motive of $X $, and hence it  is of abelian type  in $\sM_{rat}(\C)$,  if $X$ has a motive of abelian type.\par
The same result is known for a compactification $\bar \sJ ^t$ of  the twisted version $\sJ^t \to U$.This is a non trivial torsor over $\sJ_U$, and the fibers parametrise one-cycles of degree 1 on the hyperplane sections , modulo rational equivalence, see [Sa]. The  HK manifold $\bar \sJ^t$  is birational  to the OG10 manifold $\tilde M$ associated to  the Kuznetsov component  $\sA_X$, see [LPZ] .Therefore their motives are isomorphic 
\begin{equation} \label{twisted} h(\bar \sJ^t) = h( \tilde M) \end{equation}
Here $\tilde M \to M$ is  a crepant resolution   of the moduli space $M$ of $\sigma$-semistable objects in $\sA_X$with Mukai vector $\v = 2\v_0$. $\v_0$ is  a primitive element in the Mukai lattice of $\sA_X$, with $\v_0^2=2$ and   $\sigma$ is a $\v_0$-stability condition ( see [FFZ,1.7]).The Chow motive of $\tilde M$  is a direct summand of the (twisted ) motive of $h(X^5)$ , see [FFZ,Thm.1.8]. Therefore it  is of abelian type  if $X$ has a motive of abelian type and the same result  holds true for $ h(\bar \sJ^t)$ ,by \ref{twisted}.\par 
In Sect.2 we prove that the motive of the OG10  HK variety constructed in [LSV] is a direct summand of the (twisted) motive of $X^5$ and hence is of abelian type if the motive of $X$ is of abelian type.\par
In Sect.3 we show that for  smooth cubic fourfolds in a countable union of Hasset divisors $\sC_d$  the motives of  a compactification $\bar \sJ$ and of a twisted compactification $\bar\sJ^t$ coincide and belong to the subcategory generated by the motive of a K3 surface $S$.\par
In Sect.4 we consider the birational transformation $\tilde \sigma : \bar \sJ \dashrightarrow  \bar \sJ $ induced by an automorphsim $\sigma$ of $X$ and describe a family of cubics with a non-symplectic automorphism of order 3 such that $\tilde \sigma$ is a regular automorphism and the motive $h(\bar \sJ )$ is of abelian type.

\section{The motive of a general LSV}
In this section we prove the following result.
\begin{thm}\label{LSV} Let  $X$ be general cubic fourfold and let $J(X)$ be the HK manifold of OG10 type constructed in [LSV]. Then $h(\sJ(X)) $ is a direct summand of the motive
$$ \bigoplus_{1\le i \le 5}h(X^5)(l_i)   \ ; \   0 \le l_i \le 20 $$
in $\sM_{rat}(\C)$. Therefore  $\sJ(X) $ has a motive of abelian type if $h(X)$ is of abelian type\end{thm}
First we recall  some result  from [Vois].\par
Let $X \subset \P^5_{\C}$ be a  general h cubic  fourfold and let $U \subset B$,  with $B = (\P^5)^*$, the open subset parametrizing hyperplanes
$H \subset \P^5$  such that the cubic 3-fold  $Y_H =X \cap H$ is smooth. Let
$\pi: \sJ_U \to U$ be the Lagrangian fibration whose fiber over the point  $H \in U$ is the intermediate Jacobian $J(Y_H)$. Let $D \subset F(X)=F $ be the  ample uniruled divisor coming from Voisin's rational self-map
$\phi : F \dashrightarrow F$ as in [SV,Sect.18]. For every line $l \subset X$ there is $\P^3 \subset (\P^5)^*$ bundle of hyperplane sections  $Y_H =X \cap H$,  with $H \in (\P^5)^*$, containing $l$. Therefore we get a $\P^3$-bundle  $f : \sF \to F(X)$. Let $\sF_D$ be the inverse image $f^{-1}(D)$ in $\sF$. Let $\sJ(X) \to B$ be the   compactification of $\pi : J_U \to U$ constructed in [LSV] and let $Z \subset J(X) $ 
be the image of $\sF_D$ under the map
\begin{equation} \label{universal} \Psi : \sF \to J(X).\end{equation}
The map $\Psi$ sends  a line $l \subset Y_H$ , with $H \in U \subset (\P^5)^*$, to the  image, under the map
\begin{equation}\label{Jacobian}   \phi_H : CH_1(Y_H)_{hom} \to  J(Y_H)  \end{equation}
of the class $3[l] -h^2 $,where $h$ is the class of the hyperplane section. The  subvariety   $Z$ of $\sJ(X)$  has codimension 4 (dimension 6) and dominates $ B=(\P^5)^*$, see [Vois, Cor .3.5]. In the fibration 
 $$  g: Z \to B$$ 
 \noindent the fibers are  curves $\Psi(D_H)$, with $D_H = D \cap F(Y_H)$   and  $F(Y_H)$ the surface of lines.The map
 $$ \mu_{Z,U} :Z\times_U \cdots \times_U Z \to \sJ_U$$
\noindent sending $(a_1,\cdots , a_5) $   to   $ \sum_{1 \le i \le 5} a_i$ is dominant, finite  and extends to a morphism  $ Z\times_B \cdots \times_B Z \to B$, see [Ped ,Prop 3.2.].
 Therefore  the maps $\mu_{Z,U}$ induce  a generically finite  surjective morphism $\mu$   
\begin{equation} \label{morphism} \mu : Z\times_B\cdots \times_B Z\to  \sJ(X) \end{equation}
\noindent  which fits in the diagram  
   \begin{equation} \label{diagram} \CD Z\times_B \cdots \times_B Z @>{\mu_5}>> \sJ(X)\\
 @VVV      @VVV   \\
 B@>{=}>>B  \endCD \end{equation}
 \noindent see [Vois, Cor.3.4]
 \begin{lm}\label{int} Let $Y_H$ be a  smooth hyperplane section of the fourfold $X$, $F(Y_H)$  its surface of lines  and  let $D_H = F(Y) \cap D$, where $D$ is the uniruled divisor in $F(X)$. For every line $l \in D_H$ there is a finite number of lines $l_1,\cdots ,l_N$  in $D_H$ such that $l_i \cap l \ne \emptyset$.\end{lm} 
 \begin{proof} For every line $l\subset Y_H$ we have 

$$3 C_l  =g_H \in CH_1(F(Y_H)) \ ; \   (C_l)^2=5,$$  

\noindent see [Huy,1.15 and 1.17].  Here $C_l$ is the curve of lines on $Y_H$ meeting $l$ and  $g_H$ is the Pl\"ucker polarization on $Y_H$. Since the Pl\"ucker polarization $g$ on $F(X)$ restricts to the Pl\"ucker polarization 
$g\vert_{F(Y_H)}=g_H$ on $Y_H$  we get 

$$ 3n C_l =(n g)\vert_{F(Y_H)} = D_H   \ ;  \   D_H \cdot C_l = (3n) C^2_l=15 n=N $$
\noindent because $D = n g \in CH^1(F)$ (with $n=60$ ) and  $(n g)\vert_{F(Y_H)} =D_H $. \end{proof}
From  \ref{diagram} we get a map of motives $h_B( Z\times_B \cdots \times_B Z ) \to h_B(\sJ(X)))$ in the category $\sM_{rat}(B)$ of Chow motives over $B$. Objects in  $\sM_{rat}(B)$ are triples $(X,p(X/B),m)$ with $X$ over the base $B$ and $p(X/B)$ a correspondence of degree 0 in $\Corr_B(X,X)$,  see [MNP,8.1.3].  Let $\Delta_B : X \to X\times_BX$ be the diagonal embedding. Then 
\begin{equation} \label{diagonal}  h_B(X) =(X,\Delta_B,0)\end{equation}
Since the morphism  $\mu$  in  \ref{morphism}  is generically finite the motive $h_B(\sJ(X)$ is a direct summand of $h_B( Z\times_B \cdots \times_B Z $ in $\sM_{rat}(B)$,see [MNP, 2.3 (vi)].\par
\begin{lm}\label {fibrations} Let $Z \to B$ and let $\sY\to B$, with $\sY $ the universal family of hyperplane sections of the cubic fourfold $X$. Then 
$ h_B(Z\times_B\ Z\cdots \times_B Z)$ is a direct summand of $h_B(\sY\times_B \cdots \times_B \sY)$ in the category $\sM_{rat}(B)$ of Chow motives over $B$ 
  \end{lm}
\begin{proof} The fibration  $Z \to  B$ has  fibers the curves images of  $D_H$, while $\sY \to B$ has  fibers $Y_H=X \cap H$. We will define correspondences

$$\Gamma \in A_*(Z \times_B \sY)   \  , \   T \in A_*(\sY \times_B Z) ,$$
\noindent induced  by  $\{\Gamma_H \}_{H\in B}$ with $\Gamma_H \in A_*( D_H \times Y_H)$ and $T=\{T_H\}_{H\in B}$ with $T_H \in A_*(Y_H \times D_H)$.Let
$\Gamma_H\in A_1(D_H \times Y_H)$ with $\Gamma_H = \{(l,\sum_i[y_i]\} $, where $[y_i]$, for $i = \,\cdots n$, is the class in $A_0(Y_H)$ of the point $y_i =l\cap l_i$ as in \ref{int}.
Let $T_H \in A_3(Y_H \times D_H)$ be the correspondence given by the closure of  the graph of the multivalued  map sending a point  $y \in Y_H$ to the  finite set of lines in $D_H$ trough $y$.
Then
$$T_H \circ \Gamma_H=\{  [l], \sum_{1 \le i\le N }  [l_i] \} .$$
Since
$$\P^2_{<l , l_i>}\cap Y = 2l + l_i \   for  \ i = 1, \cdots, N$$
we get $ 2 [l] +  [l_i] =h^2$, with $h$ the class of a hyperplane section. Therefore 
\begin{equation} \label{hom} 2N[l] +\sum_{1 \le i \le N}[l_i] =0   \   in   \    CH_1(Y_H)_{hom}.\end {equation}
Let $\Psi: \sF \to \sJ(X)$ be the map in \ref{universal}, given by $ \psi_H : CH_1(Y_H)_{hom} \to \sJ(Y_H)$. Since $\psi_H ([l] ) = 3[l] -h^2 $ the kernel  of the restriction of $\psi_H$ to $CH_0(D_H)$
is generated by the the classes of triple lines in $D_H$, ie such  $3[l] =h^2 $.For every $H \in B$ the set of triple  lines $C_H \subset  F(Y_H) $ is finite, see [CG, Lemma 10.15]\par
We claim that for a generic $H$ in $(P^5)^* $ the curve $D_H$ does not contain any triple line , hence the cycle  $2N[l] + \sum_i [l_i]$  is not in $Ker(\psi_{H})$.\par
 Let  $V \subset  F = F(X) $ be the locus of triple lines. Then $V $is a projective irriducible surface of general type. The intersection $ C = V \cap S$,  where $S$  is the surface of second type lines, is an  irreducible curve. Its normalization 
 $\tilde C$ is  smooth and  contained in  $E_S$, see [GK, Prop.4.11]. Here $E_S$  is the exceptional  divisor of the blow-up  $\tilde F $ of $F$  along $S$ . The uniruled divisor $D \subset F$  is normalization isomorphic to $E_S$.  Let  
 $\sC \to B$  be  the universal family of triple lines ,whose fibre over $H \in B $ is the finite set of triple lines $C_H.$ Then  $\sC$    is a divisor in the fibration $\sD \to B$  with fibre  $D_H$ over $H$.  Therefore for a generic $H \in B $
 $$ D_H \cap C_H = \emptyset.$$
 Therefore the equality in \ref{hom} implies  
\begin{equation}  \label{fibres}  2N[l] +\sum_{1 \le i \le N}[l_i] =0  \  in   \   CH_0(D_H)_{hom} \end{equation}
The variety  $Z \subset \sJ(X)$ is the image of $\sF_D$  and is defined over $B=(\P^5)^* $, with fibre  the image, under the map  $\Psi$  of the curve $D_H \subset F(Y_H)$
Therefore the map $\psi_H $ when restricted to the fibers of $Z$ is the composition
$$CH_0(D_H)_{hom} \to CH_1(Y_H)_{hom} \to \sJ(Y_H). $$
From the  equality on \ref{fibres} we get
 $$\Psi_H ( 2N[l] ) =  \Psi_H(\ - \sum_i[ [l_i]).$$ 
\noindent  for a generic $H \in B$. Therefore there is an open subset $V \subset B$ such that, for every $H \in V$
$$T_H \circ (-1/2) \Gamma_H = \Delta_V(Z)  \  in   \    A^*(Z \times_V Z) = CH^*(Z \times_V Z) \otimes \Q$$
Taking the closure over $B$ of the correspondences  $\Gamma_H$ and $T_H$ we get  $\Gamma \in A^* (Z \times_B  \sY)$  and $T \in A^*(\sY \times_B Z)$ such that
\begin{equation} \label{composition} T \circ (-1/2N)\Gamma = \Delta_B(Z) \in A^*(Z \times_B Z),\end{equation}  
The correspondences $(-(1/2N)\Gamma,$ and $T $  induce morphisms in $\sM_{rat}(B) $
  $$(-1/2N)\Gamma \in \Hom_{\sM_{rat}(B)}( h_B(Z), h_B(\sY)) \ ; \   T \in Hom_{\sM_{rat}(B)}( h_B(\sY), h_B(Z)).$$
From \ref{composition} and \ref{diagonal} we get that  the map of relative motives $T: h_B(\sY) \to h_B(Z)$ has  a right inverse and hence $ h_B(Z)$ is a direct summand of $h_B( \sY)$, see [BP,Lemma2.3].\par
 Since the product in the tensor category $\sM_{rat}(B)$ is given by the fibre product over the base $B$ we get
$$h_B(Z\times_B \cdots \times_B Z)=  h_B(Z)\otimes h_B(Z)\cdots \otimes h_B(Z),$$
$$h_B(\sY\times_B\ \sY \cdots \times_B \sY)= h_B(\sY)\otimes h_B(\sY) \cdots \otimes h_B(\sY).$$
\noindent Therefore the motive $h_B(Z\times_B \cdots \times_B Z $is a direct summand of  $h_B(\sY\times_B\ \sY \cdots \times_B \sY)$ in $\sM_{rat}(B)$. \end{proof} 
 \begin{cor}The motive $h(\sJ(X)) $ is a direct summand of the motive
$$ \bigoplus_{1\le i \le 5}h(X^5)(l_i)   \ ; \   0 \le l_i \le 20 $$
\noindent in $\sM_{rat}(\C)$.
\end{cor}
\begin{proof} By \ref{diagram} the motive $h_B(\sJ(X))$ is a direct summand of $h_B(Z\times_BZ\cdots\times_B Z) $ in $\sM_{rat}(B)$  and the last one is a direct summand of $h_B(\sY\times_B\ \sY \cdots \times_B \sY)$.\par
 The morphism $B \to \Spec \C$ induces a functor $\sM_{rat}(B) \to \sM_{rat}(\C)$ between pseudo-abelian, additive, tensor categories  sending a motive $(X,\Delta_B(X),m)$ to
$(X,\Delta_X,m)$, direct sums to direct sums and tensor products to tensor products, see [MNP, 8.1.7]. Therefore 
$ h(Z^5) $ is  direct summand of $h(\sY^5) $ in $\sM_{rat}(\C)$.\par
 For every point $x \in X$ there is a $(\P^4)^* \subset (\P^5)^*$ of hyperplane sections $Y_H$ containing $x$. Therefore  
$\sY$ is a $\P^4$- bundle over $X$ and there is an isomorphism
$$h(\sY) \simeq \bigoplus_{0 \le  i \le 4} h(X)(i) $$
that implies
$$h(\sY^5) = h(\sY) \otimes \cdots \otimes h(\sY)\simeq \bigoplus_{1\le i \le 5} h(X^5)(l_i),$$
with $0\le l_i\le 20$. The motive $h(Z^5)$, being a direct  summand of $h(\sY^5)$ is  a direct summand of the motive  $M = \bigoplus_{1\le i \le 5} h(X^5)(l_i)$. Since the motive  $h( \sJ(X))$ is a direct summand of $h(Z^5)$ it is also a direct summand of  $M$.
\end{proof}
\begin{rk} (1)Theorem \ref{LSV}  gives a positive answer to a question asked in [ACLS, 9.5], where, in order to prove the Lefschetz Standard Conjecture for the compactification $\sJ(X)$ described in [LSV], the Authors asked 
whether there is an inclusion between the Chow motive of $\sJ(X)$ and the (twisted) motive of a power of $X$.\par
(2) In [ABP, Thm.5.2] it is proved that there are infinitely many Pfaffian cubic fourfold  in  $\sC_{14}$ such that the motive $h(\sJ(X))$  is of abelian type.\end{rk} 

\section{The motives of  $\bar \sJ $ and $\bar \sJ^t$}
 If $X$ is very general then the twisted manifold $\bar \sJ^t$ is not birational to the untwisted one $\bar \sJ$, but this may not be true for a special $X$ in a Hassett divisor $\sC _d$.\par
Let    $X$ be   a cubic fourfold in a Hassett divisor $\sC_d$, where $d$ satisfies the following numerical condition 
$$(**)  \exists f,g \in \Z   \   with \   g\vert 2n^2+2n+2 \  ,  \ n \in \N  \  and   \ d =f^2g $$
 There is a equivalence of categories
\begin{equation} \label{Brauer} \sA_X \simeq D^b(S,\alpha),\end{equation}
\noindent where $\sA_X$ is the Kuznetsov component of the derived category $D^b(X)$, $S$ is  a K3 surface and $\alpha$ a Brauer class in $\Br(S)$, see [Bull].
\begin{prop} Let $X\in \sC_d$,where $d$ satisfies (**) and $d\equiv 2(6)$.Then $h( \bar \sJ) = h(\bar \sJ^t)$ and there is  K3 surface $S$ such that the motives $h(X) ,h(\bar \sJ ) , h(\bar \sJ^t) $ alll belong to the pseudo-abelian tensor subcategory of $\sM_{rat}(\C)$ generated by the motive of $S$.Therefore if $S$ has a  motive of abelian type  the motives $h(X) ,h(\bar \sJ )) , h(\bar \sJ^t)$ are of abelian type.\end{prop}
\begin{proof} The  equivalence of categories in \ref{Brauer} implies that there is an isomorphism of transcendental motives, see [Bull]
$$ t(X) \simeq t_2(S)(1),$$
\noindent where the motives $h(X)$ and h(S) have the following Chow-K\"unneth decompositions  in  the (covariant) category of Chow motives. $\sM_{rat}(\C)$.
$$h(X) = \un \oplus \L \oplus (\L^2)^{\oplus \rho(X)} \oplus t(X) \oplus \L^3 \oplus \L^4.$$
$$h(S)=\un \oplus \L^{\oplus \rho(S)}\oplus t_2(S) \oplus \L^2$$
Here $\L$ is the Lefschetz motive, $\rho(X) =\rank A^2(X)$, with $A^2(X)= CH^2(X)\otimes \Q$ and $\rho(S)$ is the rank of the Neron-Severi $NS(S))$. Therefore the motive $h(X) $ belongs to the  pseudo-abelian tensor subcategory $\Mot(S)$ generated by the motive of $S$. The  manifolds $\bar \sJ $ and $\bar \sJ^t$ are birational if and only if $X \in \sC_d$ with $d \equiv 2 (6)$, in which case 
$$h( \bar \sJ)= h(\bar \sJ ^t) $$
\noindent see [BG,Rk.5.3]. Let $M$ be  the moduli space of stable objects in $D^b(S,\alpha) $ and let $\tilde M \to M$ be a crepant resolution. Then the Chow motive of $\tilde M$ belongs to the pseudo-abelian tensor subcategory $\Mot(S)$
generated by the motive of $S$, see [FFZ, Prop.4.4].\par
 The manifold $\bar \sJ^t$ is birational to a moduli space of objects of OG10 type  in $\sA_X $ and  hence,  from  \ref{Brauer}, the motive $h(\bar \sJ^t)$ is in $\Mot(S)$. 
\end{proof}
\begin{ex} Let $\sF$ be the family  of cubic fourfolds  $X \subset \P^5_{\C}$ which are invariant under the automorphisms of $\P^5$
$$\sigma : [x_0,x_1,x_2,x_3,x_4,x_5] \to :[x_0,x_1,x_2, \zeta x_3, \zeta x_4,\zeta x_5]$$
\noindent with $\zeta^3 = 1, \zeta \ne 1$.\par
Every $X \in \sF$ has an equation of the form
$$F(x_0,x_1,x_2,x_3,x_4,x_5) =f(x_0,x_1,x_2)  +  g(x_3,x_4,x_5) =0  $$
 \noindent  with   $f$  and    $g$   of  degree   3.\par 
Let $Z\subset \P^3$ and $T\subset \P^3$ be the cubic surfaces defined by $f(x_0,x_1,x_2) -t^3 =0$ and $ g(x_3,x_4,x_5)-t^3 =0$. The plane $t=0$ cuts a smooth cubic curve $C\subset Z$ and a smooth cubic curve $D \subset T$. By [CT,Prop.1.2]  there is a rational map $\P^3 \times \P^3 \to \P^5$ which induces a rational dominant map $\psi : Z \times T \to X$ and whose locus of indeterminacy is  $C\times D$. Let $\tilde X$ be the blow-up of $Z \times T$ at $C\times D$. The fourfold $X$ contains two disjoint planes $P_1$ and $P_2$, see [CT, Rk.2.4], hence it is rational  and belongs to $\sC_{14}$. There is an isomorphism
\begin{equation} \label {iso} h( \bar \sJ)= h(\bar \sJ ^t) \end{equation}
\noindent and $h(\bar \sJ ^t) $ lies in the subcategory of $\sM_{rat}(\C) $ generated by a K3 surface $S$.In this case  the motives of $S$ and of $X$ are both of abelian type. By  Manin's formula there is an isomorphism

$$h(\tilde X)\simeq h(Z\times T) \oplus h (C \times D(1) .$$

\noindent The motive of the fourfold $Z \times T$ is finite dimensional and has no transcendental  part, since both the surfaces $Z$ and $T$ are rational. Therefore the motive $h(\tilde X)$ is finite dimensional and of abelian type  because   its transcendental part coincides with the transcendental motive $t_2(C \times D)(1)$, which lies in the subcategory of $\sM_{rat}$ generated by the motives of curves. The map $\psi$ induces a finite morphism $\tilde \psi: \tilde X \to X$ and hence $h(X)$ is a direct summand of $h(\tilde X)$. It follows that  also $h(X)$ is finite dimensional and  of abelian type. By the results in [Bull] there is an isomophism of motives 
$$t_2(S) (1) \simeq t(X),$$
\noindent hence also $h(S)$ is of abelian type. By \ref{iso} the motives of the compactifications $\bar \sJ$ and $\bar\sJ^t$ are of abelian type.
\end{ex}
\section{Automorphisms of LSV manifolds}
Let $X \subset \P^5$ be a smooth cubic fourfold and let $\bar\sJ$ be a  compactification of $\pi : \sJ_U \to U$, where $U \subset (P^5)^*$ is such that 
$Y_H = X\cap H $ is smooth, for $H \in U$. Any automorphism  $\sigma$  of $X$ induces an automorphism  $\tilde \sigma$ on the open subset  $ \sJ_U \subset \bar\sJ$,.Therefore every automorphism $ \sigma$ of $X$ induces a birational transformation of  $\bar\sJ$.The fibres of $\sJ_U \to U$ are given by the intermediate Jacobians $J(Y_H) $ and the automorphism $\tilde \sigma$  of $\sJ_U $ will either permute fibres  or induce an automorphism of an invariant fibre.The automorphism $\sigma$ of $X$  is symplectic if and only if the induced birational transformation $\tilde \sigma$  is symplectic, see[BG.Lemma 5.7]. \par 
 The involution $x \to -x$ on each fibre of $\sJ_U \to U$  defines  an involution $\tau \in \Bir(\bar\sJ)$ which commutes with all the birational transformations $\tilde \sigma$  induced from  an automorphism $ \sigma$  of $X$ on $\sJ_U$.\par 
The following result has been proved in [MM, Prop.6.4]
\begin{prop} Let $X \subset \P^5$ be a smooth cubic fourfold and $G $ be a finite subgroupo of symplectic automorphisms   of $X$. Then the group of symplectic birational transformations of any compactification $\bar\sJ$  contains a finite subgroup isomorphic to $G$. \end{prop}
 Since birational hyper-K\"alher varieties have isomorphic motives, the  birational transformation  $\tilde \sigma$   induces an  automorphism of the Chow motive $h(\bar \sJ)$.\par
A sufficient condition for $\sigma^*$ being a regular automorphism has been proved in [Sa, Prop.3.10]. 
\begin{prop} \label{Sacca'}Let $X$  be a smooth cubic fourfold  and let  $\bar \sJ$ be a  smooth compacticfication of $\pi:  \sJ_U \to U$, such that the fibres of $\pi :\bar \sJ \to  B= (\P^5)^* $ are irreducible. Then any birationall automorphism 
$ \tau : \bar \sJ \dashrightarrow \bar\sJ$ which fixes $L = \pi^*(\sO_{\P^5}(1))$ extends to a regular automorphism.\end{prop}
The above result notably applies to the compactification $\sJ(X)$ in [LSV], for $X$ general,  and to the compactification $\bar \sJ$, for a very good cubic fourfold. Since in both cases the fibers of $\pi$ are irreducible  a   birational automorphism $\tilde \sigma$ which fixes $\pi^*( \sO_{\P^5}(1))$  is a regular automorphism  of the hyper-K\"alher  compactification .\par
Giovenzana-Grossi-Onorati-Veniani in [GGOV] proved that any symplectic automorphism $f$ of finite order on a HK variety, $Z$ deformation equivalent to OG10, is trivial. They show that $f$ acts trivially on the second cohomology
group $H^2(Z,\Z) $ and thus is the identity.\par
In particular a symplectic automorphism  $\sigma$ of $X $ does not extend to a regular automorphism $\tilde \sigma$ of a smooth compactification $\bar \sJ$ of $\pi : \sJ_U  \to U$.\par
On the other hand S.Billi and A.Grossi showed that, if $\sigma$ is a non-symplectic automorphism of prime order  on a cubic fourfold $X$, then the induced birational transformation  $\tilde \sigma \in \Bir(\bar \sJ)$ is an automorphism, see [BG, Lemma 5.9]. Their proof is based on the description of the action of $\sigma$ on the lattice $H^2(\bar \sJ,\Z)$ showing that
$$H^2(\bar\sJ, \Z)^{\tilde \sigma}=NS(\bar \sJ.$$
\begin{ex} Here we describe a 10-dimensional family $V$ of cubic fourfolds with a non-symplectic automorphism of order 3 with the following properties\par
(1) For a general  $X\in V$ the compactification $ \bar \sJ$ of the Lagrangian fibration $\sJ_U \to U$ is unique;\par
(2) The   birational automorphism $\tilde \sigma : \bar \sJ  \dashrightarrow  \bar \sJ$ is  a regular automorphism ;\par  
(3) The  Chow motive $h(\bar \sJ ) \in \sM_{rat}$ is of abelian type.\par 
Let $\sigma$ be the order 3 automorphism
$$ \sigma : [x_0,x_1,x_2,x_3,x_4,x_5] \to [x_0,x_1,x_2,x_3,x_4, \zeta x_5],$$
\noindent with $\zeta$ a primitive cubic root of 1. Every $X$ belonging to the family $V$ of cubic fourfolds invariant under $\sigma$ has an equation of the form
$$ F(x_0,x_1,x_2,x_3,x_4) + x^3_5 =0 $$
\noindent with $ F$ homogeneous of  degree 3.  For a general element $X \in V$  the algebraic lattice $A(X) $ has rank 1, see [BG, Table 3].  From the isomorphism  of rational Hodge structures

$$H^2_{tr}(\bar \sJ,\Q) \simeq H^4_{tr}(X,\Q)$$

\noindent we get  $\rho(\bar \sJ ) = \rank H^{2,2}(X,\Q)) +1 =2$, see [Sa, Lemma 3.2].  Any hyper-K\"alher compactification  $\bar \sJ$  of $\sJ_U \to U$ is isomorphic to the compactification  $\sJ(X)$ constructed in [LSV], see [Sa, Cor.3.4]. 
Therefore the fibres of $\bar \sJ \to (\P^5)^*$ are irreducible.The Neron -Severi  lattice $NS(\bar \sJ)$ is generated by $L = \pi^*\sO_{\P^5} (1)$ and the relative  theta divisor $\Theta$, see [Sa ,Rk 3.1]. 
Since $NS( \bar \sJ) = H^2(\bar \sJ,\Q)^{\tilde \sigma}$ the divisor $L$ is fixed by $\tilde \sigma$. From \ref{Sacca'} we get that $\tilde \sigma : \bar \sJ  \dashrightarrow \bar \sJ$  is  a regular automorphism.\par 
Every cubic fourfold $X \in V $ has a motiive of abelian type (see [Lat]) and  $h(\bar \sJ) =h( \sJ(X))$  belongs to the subcategory generated by the motive of $X$, by \ref{LSV} . Therefore $h(\bar \sJ)$ 
is of abelian type, which proves (3).
 \end{ex}

\end{document}